\documentclass[paper=A4,abstract=true,fontsize=11pt]{scrartcl}
\usepackage{amsfonts,amsmath,amsthm,amssymb}
\usepackage{url}
\usepackage{color}
\usepackage{graphicx}
\newcommand{\R}{\mathbb{R}}
\newcommand{\N}{\mathbb{N}}

\newcommand{\e}{\varepsilon}
\renewcommand{\d}{\delta}

\newcommand{\diam}{\operatorname{diam}}
\newcommand{\comment}[1]{}	%

\newcommand{\vol}{\operatorname{vol}}

\newtheorem{theorem}{Theorem}
\newtheorem{cor}[theorem]{Corollary}
\newtheorem{lemma}[theorem]{Lemma}
\newtheorem{prop}[theorem]{Proposition}

\setkomafont{captionlabel}{\bfseries}
\title{An upper bound on the volume of the symmetric difference of a body and a congruent copy\footnote{This research was supported by the Deutsche Forschungsgemeinschaft within the research training group `Methods for Discrete Structures' (GRK 1408).}}
\author{Daria Schymura\footnote{schymura@mi.fu-berlin.de}}
\date{Institut f\"ur Informatik, Freie Universit\"at Berlin}

\begin{document}
\maketitle

\begin{abstract}
Let $A$ be a bounded subset of $\R^d$. We give an upper bound on the volume of the symmetric difference of $A$ and $f(A)$ where $f$ is a translation, a rotation, or the composition of both, a rigid motion. The volume is measured by the $d$-dimensional Hausdorff measure, which coincides with the Lebesgue measure for Lebesgue measurable sets.

We bound the volume of the symmetric difference of $A$ and $f(A)$ in terms of the $(d-1)$-dimensional volume of the boundary of $A$ and the maximal distance of a boundary point to its image under $f$. The boundary is measured by the $(d-1)$-dimensional Hausdorff measure, which matches the surface area for sufficiently nice sets. 
In the case of translations, our bound is sharp. In the case of rotations, we get a sharp bound under the assumption that the boundary is sufficiently nice.

%

The motivation to study these bounds comes from shape matching. For two shapes $A$ and $B$ in $\R^d$ and a class of transformations, the matching problem asks for a transformation $f$ such that $f(A)$ and $B$ match optimally. The quality of the match is measured by some similarity measure, for instance the volume of overlap.

Let $A$ and $B$ be bounded subsets of $\R^d$, and let $F$ be the function that maps a rigid motion~$r$ to the volume of overlap of $r(A)$ and $B$. Maximizing this function is a shape matching problem, and knowing that $F$ is Lipschitz continuous helps to solve it. We apply our results to bound the difference $|F(r) - F(s)|$ for rigid motions $r,s$ that are close, implying that $F$ is Lipschitz continuous for many metrics on the space of rigid motions. Depending on the metric, also a Lipschitz constant can be deduced from the bound.
\end{abstract}

\section{Introduction}

Let $K \subset \R^d$ be a compact, convex body and $t \in \R^d$ a translation vector. Let $K \oplus [0,1] t$ be the set of all points $k + \lambda t$ where $k \in K$ and $\lambda \in [0,1]$, which is the set that is 'swept over' by $K$ when translating $K$ to $K+t$. Denote the projection of $K$ to the orthogonal space of $t$ by $K|t^\bot$, and let $|t|$ be the Euclidean norm of $t$. See Figure \ref{fig:convex}.

\begin{figure}[ht]
\begin{center}
\includegraphics[width=0.35\textwidth]{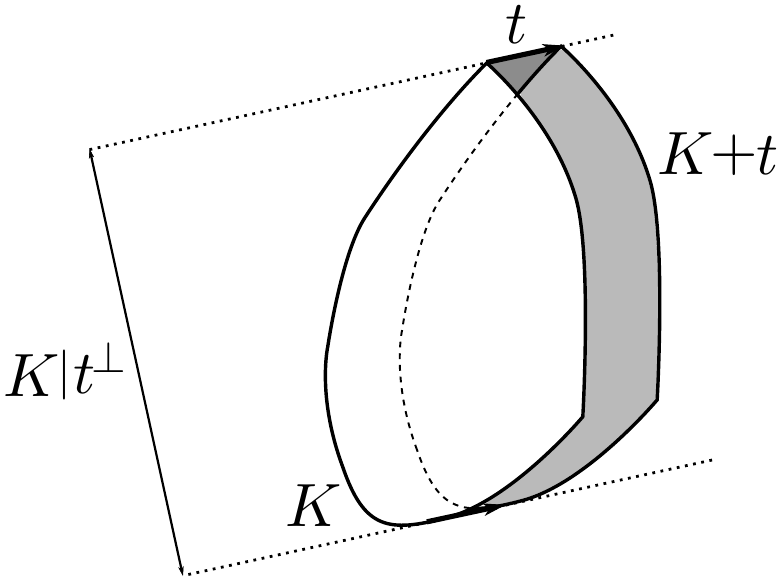}
\caption{A convex body $K$ and a translated copy $K+t$. The set $(K+t)\setminus K$ is drawn in light-gray. Both gray regions together form $(K \oplus [0,1]t) \setminus K$.}
\end{center}
\label{fig:convex}
\end{figure}
The volume of the set that is 'swept over' by $K$ when translating $K$ to $K+t$ can be computed by Cavalieri's principle as 
\begin{equation}
\operatorname{vol}_d(K \oplus [0,1]t) = \operatorname{vol}_d(K) + |t| \, \operatorname{vol}_{d-1}(K |t^{\bot}).
\label{eqn:cavalieri} 
\end{equation}
See for example \cite[Appendix A.5]{gardner}. The volume of $(K \oplus [0,1]t) \setminus K$ can be computed with the above formula; it is an upper bound on the volume of the set $(K+t) \setminus K$ since it is obviously contained in $(K \oplus [0,1]t) \setminus K$. 

The symmetric difference of two sets is defined as $A \bigtriangleup B = (A \setminus B) \cup (B \setminus A)$. If $A$ and $B$ have the same volume, we have $\operatorname{vol}(A \setminus B)= \operatorname{vol}(B \setminus A)$ and therefore $\operatorname{vol}(A \bigtriangleup B) = 2 \operatorname{vol}(A \setminus B)$. So, we also have an upper bound on the volume of $K \bigtriangleup (K+t)$.

What happens if $K$ is not convex? We do not ask for an exact formula, but for an upper bound on the volume of $K \bigtriangleup (K+t)$. Figure \ref{fig:notConvex} shows a comb-like body $E$, which indicates that such an upper bound in terms of the length of the translation vector and the $(d-1)$-dimensional volume of the projection of the body cannot exist. The more teeth the comb has, the larger is the volume of $E \bigtriangleup (E+t)$. But the $(d-1)$-dimensional volume of the projection does not change if the translation vector is orthogonal to the teeth. The example suggests that an upper bound in terms of the length of the translation vector and the $(d-1)$-dimensional volume of the boundary might exist.
\begin{figure}[ht]
\begin{center}
\includegraphics[width=0.35\textwidth]{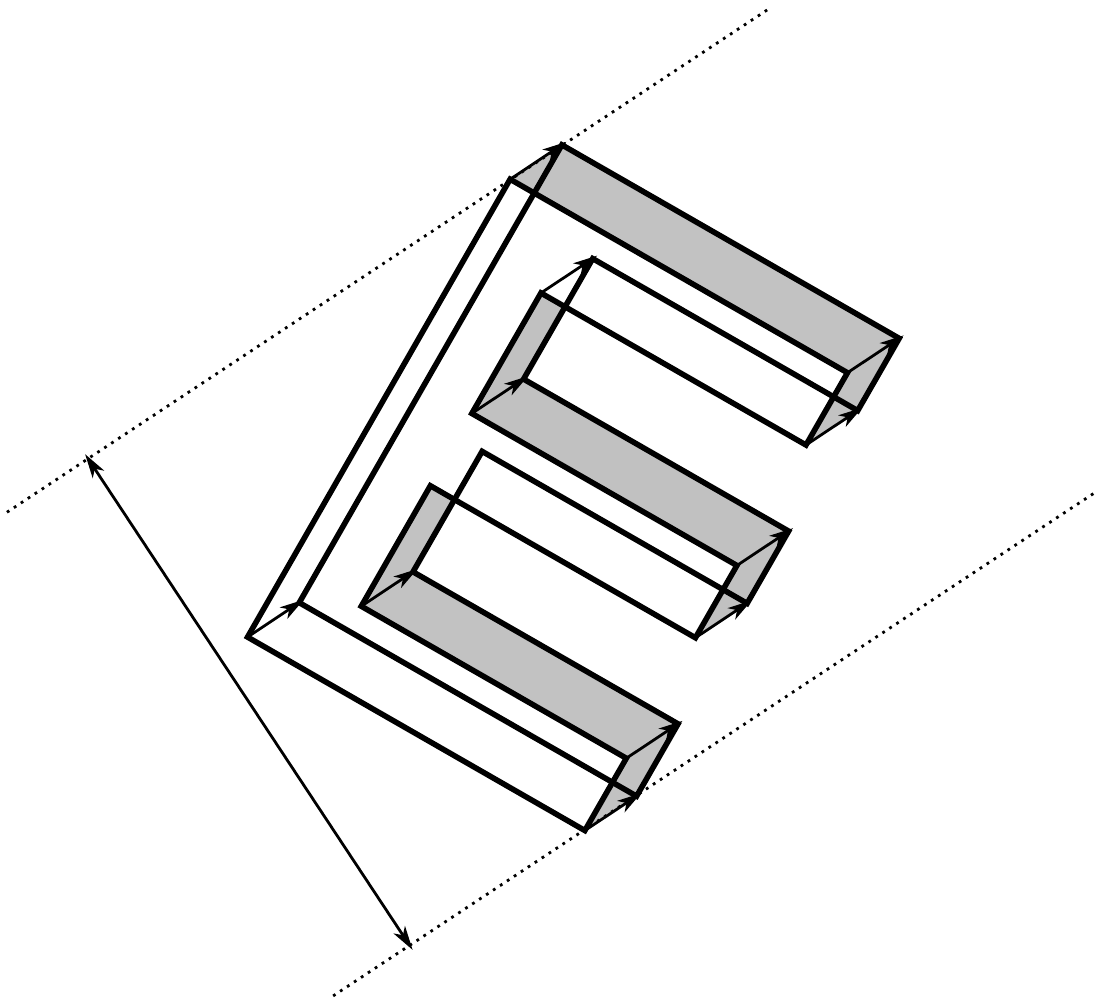}
\caption{A comb-like body $E$ and a translated copy. The set $(E \oplus [0,1]t)\setminus E)$ is shaded.}
\end{center}
\label{fig:notConvex}
\end{figure}

For a set $A \subset \R^d$, the boundary $\partial A$ is defined as the set of points that are in its closure $\operatorname{cl}(A)$, but not in its interior $\operatorname{int}(A)$. We will prove the following theorem, which shows that such an upper bound on the volume of $A \bigtriangleup (A+t)$ in terms of the $(d-1)$-dimensional volume of $\partial A$ and the length of $t$ indeed exists. Therein, $\mathcal{H}^{k}$ denotes the $k$-dimensional Hausdorff measure. For Lebesgue measurable sets, the $d$-dimensional Hausdorff measure coincides with the usual Lebesgue measure. For sets in $\R^d$ with sufficiently nice boundaries, the $(d-1)$-dimensional Hausdorff measure is the same as the intuitive surface area. The Hausdorff measure will be defined in Section \ref{sec:hausdorff}.
\begin{theorem}\label{thm:volSymDiffTrans}
Let $A \subset \R^d$ be a bounded set. Let $t \in \R^d$ be a translation vector. Then, $$\mathcal{H}^d(A \bigtriangleup (A+t)) \le |t| \, \mathcal{H}^{d-1}(\partial A).$$
\end{theorem}
This inequality is best possible, in the sense that it becomes false when the right hand side is multiplied with any constant that is smaller than one. Let us assume on the contrary that the upper bound could be multiplied with $(1-\e)$ for some small, positive~$\e$. Translate a rectangle $R$ with side lengths $1$ and $\e$ in direction of the side of length~$\e$. If the length of the translation is $\e/2$, the volume of $R \bigtriangleup (R+t)$ equals $\e$, but the modified bound gives $\e - \e^3$.

We also ask how the volume of the symmetric difference behaves when we rotate the set, instead of translating it. For a rotation matrix $M \in \R^{d \times d}$, we give an upper bound on the volume of $A \bigtriangleup MA$, in terms of the $(d-1)$-dimensional volume of the boundary of $A$ and a parameter $w$ that measures the closeness of $M$ and the identity matrix with respect to $A$. The parameter $w$ is the maximal distance between $a$ and $Ma$ among all points $a \in \partial A$.

\begin{theorem}\label{thm:volSymDiffRot}
Let $A \subset \R^d$ be a bounded set. 
Let $M \in \R^{d \times d}$ be a rotation matrix and let $w = \max_{a \in \partial A} |a - Ma|$. Then, $$\mathcal{H}^d(A \bigtriangleup MA) \le \left( \frac {2d} {d+1} \right)^{\frac{d-1}{2}} \, w \, \mathcal{H}^{d-1}(\partial A).$$
\end{theorem}
We also show that the constant $\left( \frac {2d} {d+1} \right)^{\frac{d-1}{2}}$ can be replaced by $1$ for sets that have a $(\mathcal{H}^{d-1},d-\nolinebreak 1)$-rectifiable boundary. The definition of $(\mathcal{H}^{d-1},d-1)$-rectifiable is postponed to Section \ref{sec:hausdorff}. Again, $1$ is the best possible constant because a rotation is very close to a translation if the rotation center is far away from the rotated set.

A rigid motion is the composition of a rotation and a translation. Let $SO(d) \subset \R^{d\times d}$ be the special orthogonal group that is the group of rotation matrices. Parametrize the space of rigid motions as $\mathcal{R}=SO(d) \times \R^d$ where $(M,t)\in \mathcal{R}$ denotes the rigid motion $x \mapsto Mx +t$.

Since the symmetric difference fulfills the triangle inequality, we get the following corollary for rigid motions. We assume that $\partial A$ is $(\mathcal{H}^{d-1},d-1)$-rectifiable such that we can use Theorem \ref{thm:volSymDiffRot} with constant 1, as mentioned above.

\begin{cor}\label{cor:volSymDiffRM}
Let $A \subset \R^d$ be a bounded set. Let $r=(M,t) \in \mathcal{R}$ be a rigid motion, and let $w = \max_{a \in \partial A} |a - Ma|$. If $\partial A$ is $(\mathcal{H}^{d-1},d-1)$-rectifiable, then, $$\mathcal{H}^d(A \bigtriangleup r(A)) \le \, (|t| + w) \, \mathcal{H}^{d-1}(\partial A).$$
\end{cor}

\vspace{3mm}

Apart from the fact that studying the volume of $A \bigtriangleup f(A)$ for a translation or rotation $f$ is an interesting mathematical problem on its own, we have the following motivation from shape matching for doing so. For fixed bounded sets $A,B \subset \R^d$, let $F$ be the function that maps a rigid motion $r$ to the volume of $r(A) \cap B$. The volume of overlap measures the similarity between sets. Given shapes $A, B \subset \R^d$, computing a translation or rigid motion $r$ that maximizes $F$ means finding an optimal match w.r.t. this similarity measure. See the references in \cite{regionMatching} for more literature on maximizing the volume of overlap of two shapes $A$ and $B$; additionally, see \cite{brassOverlap} and \cite{vigneron}.

In order to develop algorithms for this problem and analyze them, it is useful to prove that $F$ is Lipschitz continuous and to be able to compute a Lipschitz constant of $F$ for given $A$ and $B$ \cite{regionMatching}. We will apply our results to bound $|F(r)-F(s)|$ by $\mathcal{H}^{d-1}(\partial A)$ times a factor that is related to $\max_{a \in \partial A} | r(a) - s(a) |$.

Equip the space of rigid motions $\mathcal{R}$ with any metric that is induced by a norm on $\R^{d \times d} \times \R^d$. Then, this bound on $|F(r) - F(s)|$ implies that $F$ is Lipschitz continuous if $\partial A$ has a finite $(d-1)$-dimensional volume, and also a Lipschitz constant can be deduced.

\begin{cor}\label{cor:shapeMatching}
Let $A,B \subset \R^d$ be bounded and $\mathcal{H}^d$-measurable sets such that $\partial A$ is $(\mathcal{H}^{d-1},d-\nolinebreak 1)$-rectifiable.
Let $r=(M,p)$ and $s=(N,q)$ be rigid motions. Then, $$|F(r)-F(s)| \le \frac 1 2 \Bigl( |p-q| + w \Bigr) \, \mathcal{H}^{d-1}(\partial A) $$ where $w=\max_{a \in \partial A} |M a - N a|$.
\end{cor}

The volume of $(A+t) \setminus A$, which equals half of the volume of $(A +t) \bigtriangleup A$, arises also in other contexts.
For $A \subset \R^d$, the function $g_A$ that maps a translation vector $t \in \R^d$ to the volume of $(A+t) \cap A$ is called \emph{covariogram of $A$}, sometimes also \emph{set covariance}, and was introduced in \cite{matheron} for compact sets.

Since $g_A(0) - g_A(t) = \vol_d((A+t) \setminus A)$, this volume is related to estimating the directional derivatives of $g_A$ at $0$. For convex, compact sets $A$, these are determined in \cite{matheron}. For $u \in \mathcal{S}^{d-1}$, consider the function $\lambda \mapsto g_A(\lambda u)$ for $\lambda \in \R$; it has a continuous derivative that equals $- \vol_{d-1} (A|u^\bot)$. That $- \vol_{d-1} (A|u^\bot)$ is an upper bound on the derivative can be seen immediately from Equation (\ref{eqn:cavalieri}).

Galerne \cite{galerne} studies $g_A$ for measurable sets $A$ of finite Lebesgue measure. He computes the directional derivatives at the origin and proves that the perimeter of $A$ can be computed from these derivatives. The perimeter of a set is at most $\mathcal{H}^{d-1}(\partial A)$. He also computes the Lipschitz constant of $g_A$ in terms of the directional variation. For further details and definitions, we refer the reader to the paper and the references cited therein.

The inverse question whether the covariogram determines a convex body in $\R^d$, up to translations and reflections in the origin, is answered affirmative for the planar case in \cite{bianchi2D}. For three dimensions, convex polytopes are determined by their covariogram \cite{bianchi3D}. In dimension $\ge 4$, the question has a negative answer \cite{bianchi4D}. In the planar case, the class of sets among which the covariogram of a convex body is unique, is extended in \cite{bianchiMore2D}. 

The function $g_{B,A}(t):=\vol_d((A+t) \cap B)$ is called \emph{cross covariogram} for convex sets $A$ and $B$. Bianchi \cite{bianchiCross2D} proves that, for convex polygons $A$ and $B$ in the plane, $g_{B,A}$ determines the pair $(A,B)$, except for a few pairs of parallelograms. The family of exceptions is completely described.
For further references on and other occurences of the covariogram problem, see \cite{bianchi2D}.

\vspace{3mm}

In Section \ref{sec:hausdorff}, we introduce some notation, define the Hausdorff measure, the spherical measure and cite some properties of them. Sections \ref{sec:lineSegments} and \ref{sec:boundOnVolume} contain the proofs of Theorems \ref{thm:volSymDiffTrans} and \ref{thm:volSymDiffRot}. More precisely, we show that $A \bigtriangleup (A+t)$ and $A \bigtriangleup MA$ are contained in certain unions of line segments in Section \ref{sec:lineSegments}. For translations $t$, we show that $A \bigtriangleup (A+t) \subseteq \partial A \oplus [0,1]t$. For a rotation matrix $M$, we prove that $A \bigtriangleup MA$ is contained in the set of all line segments from $a$ to $Ma$ for $a \in \partial A$, see Figure \ref{fig:lineSegments}. We bound the volume of the unions of these line segments in Section \ref{sec:boundOnVolume}. In Section \ref{sec:shapeMatching}, we prove Corollary \ref{cor:shapeMatching}.

\section{Preliminaries}\label{sec:hausdorff}

We assign a volume to measurable subsets $A$ of $\R^d$ by the usual Lebesgue measure and denote it by $\mathcal{L}^d(A)$. Denote by $\omega_d$ the volume of the Euclidean unit ball in $\R^d$. We want to measure not only the volume of sets, but also the surface area, or $(d-1)$-dimensional volume, of their boundaries. We now define the Hausdorff measure, which generalizes the Lebesgue measure and which we will use for measuring boundaries. The following definitions of the Hausdorff measure, the spherical measure, their properties, and rectifiability can be found in \cite{federer}.

For $A \subset \R^d$, $0\le k \le d$ and $\d >0$, let $\mathcal{H}^k_\d(A)$ be the \emph{size $\d$ approximating $k$-dimensional Hausdorff measure} of $A$ that is defined as follows $$\mathcal{H}^k_\d(A) = \omega_d \, 2^{-d} \, \inf \left\{ \sum_{j=0}^\infty (\diam B_j)^k : \forall j\in \N \  B_j \subset \R^d, \diam B_j \le \d, A \subset \bigcup_{j=0}^\infty B_j \right\}.$$ We abbreviate $\d \to 0$ and $\d >0$ by $\d \to +0$. The $k$-dimensional Hausdorff measure of $A$ is then defined as $$\mathcal{H}^k(A) = \lim_{\d \to +0} \mathcal{H}^k_\d(A).$$ The limit exists and equals the supremum because $\mathcal{H}^k_\d(A) \le \mathcal{H}^k_\eta(A)$ for $\eta \le \d$, but it might equal $\infty$.

A set $A\subset \R^d$ is \emph{$\mathcal{H}^k$-measurable} if it satisfies the Carath\'eodory property, meaning that for all sets $B \subset \R^d$, we have $\mathcal{H}^k(A) =\mathcal{H}^k(A \cap B) + \mathcal{H}^k(A \setminus B)$. The Hausdorff measure is defined for all $A \subset \R^d$, but, of course, it is more meaningful for measurable sets.

The $0$-dimensional Hausdorff measure is the counting measure that gives for finite sets the number of elements and $\infty$ for infinite sets. The $(d-1)$-dimensional Hausdorff measure of sufficiently nice $(d-1)$-dimensional sets, for example smooth manifolds, equals the surface area. We measure the boundaries of sets in $\R^d$ by the $(d-1)$-dimensional Hausdorff measure. For Lebesgue measurable sets, the $d$-dimensional Hausdorff measure on $\R^d$ equals the Lebesgue measure. Recall that the Lebesgue measure is invariant under rotation and translation, and note that the Hausdorff measure is invariant under rotation and translation, too.

The sets $B_j$ in the definition of the approximating Hausdorff measure can be assumed to be convex because taking the convex hull does not increase the diameter of a set. Since convex sets are Lebesgue measurable \cite{convexMeasurable}, the sets $B_j$ can also be assumed to be Lebesgue measurable.

If we restrict the coverings $\{ B_j\}_{j \ge 0}$ in the definition of the approximating Hausdorff measure to be families of balls, then the resulting measure is called \emph{spherical measure}. For $A \subset \R^d$, we denote the $k$-dimensional spherical measure of $A$ by $\mathcal{S}^k(A)$. Since the choice of coverings is restricted, we have $\mathcal{H}^k(A) \le \mathcal{S}^k(A)$.

Jung's theorem, which we cite from \cite{federer}, gives a sharp bound on the radius of the smallest enclosing ball of a set of a fixed diameter. For regular, full-dimensional simplices, equality holds.

\begin{theorem}[Jung's theorem]\label{thm:Jung}
If $S \subset \R^d$ and $0 < \diam(S) < \infty$, then $S$ is contained in a unique closed ball with minimal radius, which does not exceed $\sqrt{\frac d {2d +2}} \diam(S)$.
\end{theorem}
From this, it follows that we have $\mathcal{S}^k(A) \le \bigl( \frac {2d} {2d +2}\bigr)^{k/2} \mathcal{H}^k(A)$ for all $A\subset \R^d$. In general, the Hausdorff measure and the spherical measure are not equal, but for $(\mathcal{H}^k,k)$-rectifiable subsets of $\R^d$ they agree \cite[Theorem 3.2.26]{federer}. We define the notion of rectifiability now.

A subset $E$ of a metric space $X$ is \emph{$k$-rectifiable} if there exists a Lipschitz continuous function that maps some bounded subset of $\R^k$ onto $E$. The union of countably many $k$-rectifiable sets is called \emph{countably $k$-rectifiable}. $E$ is called \emph{countably $(\mu,k)$-rectifiable} if $\mu$ is a measure defined on $E$ and there is a countably $k$-rectifiable set that contains $\mu$-almost all of $E$. If, additionally, $\mu(E) < \infty$, then $E$ is called \emph{$(\mu,k)$-rectifiable}.

Lastly, we cite the isodiametric inequality from \cite{federer}. It says that, among the Lebesgue measurable sets of a fixed diameter, Euclidean balls have the largest volume.

\begin{theorem}[Isodiametric Inequality] \label{thm:isodiametricIneq}
If $\emptyset \neq S \subset \R^d$ is Lebesgue measurable, then $\mathcal{L}^d (S) \le \omega_d 2^{-d} (\diam(S))^d$.
\end{theorem}


\section{Covering the symmetric difference of a body and a copy by line segments}\label{sec:lineSegments}

For $x,y \in \R^d$, the line segment from $x$ to $y$ is the set $\{ (1-\lambda) x + \lambda y : \lambda \in [0,1] \}$, and is denoted by $\ell(x,y)$. The Minkowski sum of two sets $A \oplus B$ equals the set of all sums $a+b$ for $a \in A$ and $b\in B$.

We show that for any bounded set $A \subset \R^d$ and a translation $t$, the set $A \bigtriangleup (A+t)$ is covered by the union of the line segments $\ell(a,a+t)$ where $a$ is in the boundary of $A$. For a rotation matrix $M$, the set $A\bigtriangleup MA$ is covered by the union of the  line segments $\ell(a,Ma)$ where $a$ again is in the boundary of $A$. See Figure \ref{fig:lineSegments}.

\begin{figure}[ht]
\begin{center}
\includegraphics[width=0.7\textwidth]{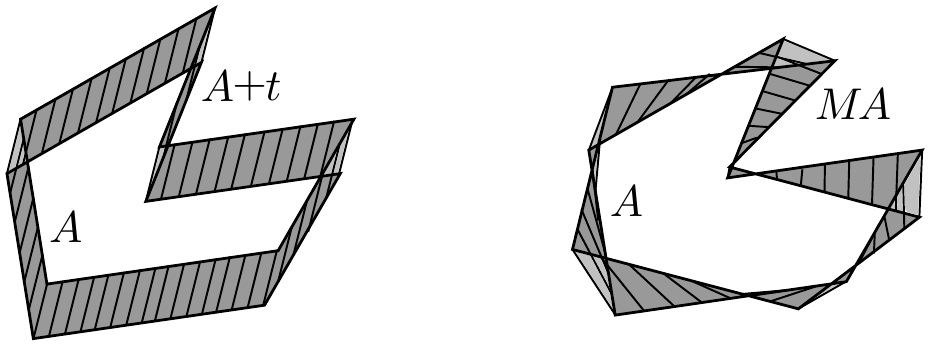}
\caption{On the left, the figure shows a body $A$ and a translated copy. On the right, the figure shows a body $A$ and a rotated copy. The symmetric differences are drawn in dark-gray. Examples of the line segments are drawn and the union of the line segments is drawn in light-gray.}
\label{fig:lineSegments}
\end{center}
\end{figure}
\begin{lemma}\label{lem:CoveringWithLineSegmentsT}
Let $A \subset \R^d$ be a bounded set, and let $t \in \R^{d}$ be a translation vector. Then, $$A \bigtriangleup (A+t) \subseteq \bigcup \{ \ell(a,a+t) : a \in \partial A \}.$$
\end{lemma}
\begin{proof}
We show that, for any translation vector $t \in \R^d$, we have $A \setminus (A+t) \subseteq \bigcup \{ \ell(a,a+t) : a \in \partial A \}$, which implies the claim. Let $a \in A \setminus (A+t)$ and let $l$ be the line $\{ a+\lambda t : \lambda \in \R \}$. If $a \in \partial A$, we are done. Otherwise $a \in \operatorname{int}(A)$ and therefore $a+t \in \operatorname{int}(A+t)$. Since $l$ intersects $\operatorname{int}(A+t)$ and $A+t$ is bounded, we have $\partial (A+t) \cap l \neq \emptyset$. Let $\lambda \in (0,1]$ such that $a+\lambda t \in \partial (A+t)$. Then $a'=a +(\lambda - 1)t$ is a point in $\partial A$ such that $a'+(1-\lambda)t=a$.
\end{proof}
\begin{lemma}\label{lem:CoveringWithLineSegmentsR}
Let $A \subset \R^d$ be a bounded set, and let $M \in \R^{d \times d}$ be a rotation matrix. Then, $$A \bigtriangleup MA \subseteq \bigcup \{ \ell(a,Ma) : a \in \partial A \}.$$
\end{lemma}
\begin{proof}
Consider the continuous function $\varphi: [0,1] \times \R^d \to \R^d$ that is defined by $\varphi(\lambda,x) = \varphi_\lambda(x) = (1-\lambda) x + \lambda Mx$. We want to show that $MA \setminus A \subseteq \varphi([0,1] \times \partial A)$, which by symmetry implies the claim. We first prove that $\varphi_\lambda$ is injective for all $\lambda \in [0,1] \setminus \{ \frac{1}{2} \}$. This implies that for each $\lambda \in [0,1] \setminus \{ \frac{1}{2} \}$ and each bounded set $S \subset \R^d$, the function $\varphi_\lambda:\operatorname{cl}(S) \to \varphi_\lambda(\operatorname{cl}(S))$ is a homeomorphism because it is bijective and linear.

Assume that there exist $x,y \in \R^d$, $x \neq y,$ such that $\varphi_\lambda(x)= \varphi_\lambda(y)$. Then $M(x-y)= (\lambda-1)/\lambda (x-y)$, so $(\lambda-1)/\lambda$ is an eigenvalue of the rotation $M$. Since for a rotation only $1$ or $-1$ can occur as eigenvalues, we get $\lambda=1/2$.

Let $y \in MA \setminus A$. We now distinguish two cases.
\begin{itemize}
\item \textbf{Case 1. }$\varphi_{1/2}:\operatorname{cl}(A) \to \varphi_{1/2}(\operatorname{cl}(A))$ is not bijective and $y \in \varphi_{1/2}(\operatorname{cl}(A))$\vspace{2mm}\\ 
Let $a,b \in \operatorname{cl}(A)$ such that $a \neq b$ and $\varphi_\lambda(a)= \varphi_\lambda(b)$, and let $x \in \varphi_{1/2}^{-1}(y)$. For each point $v$ on the line $L=\{ x + \lambda(b-a) : \lambda \in  \R \}$, we have $\varphi_{1/2}(v)= y$, due to the linearity of $\varphi_\lambda$. Since $A$ is bounded, $\partial A \cap L \neq \emptyset$ and for every point $x'$ in this set $y=\varphi(1/2,x') \in \varphi([0,1] \times \partial A)$. 
\item \textbf{Case 2. }$\varphi_{1/2}:\operatorname{cl}(A) \to \varphi_{1/2}(\operatorname{cl}(A))$ is bijective or $y \notin \varphi_{1/2}(\operatorname{cl}(A))$\vspace{2mm}\\ 
Since $y \in \varphi_1(A)$, we can define $t=\inf \{ \lambda \in [0,1] : y \in \varphi_\lambda(\operatorname{cl}(A))\}$. We now show that $y \in \varphi_t(\operatorname{cl}(A))$. Assume that $y \notin \varphi_t(\operatorname{cl}(A))$. Then for all $a \in \operatorname{cl}(A)$ the distance $|\varphi_t(a)-y|>0$. Since $\operatorname{cl}(A)$ is compact and the distance is continuous, we would have $\min \{ |\varphi_t(a)-y| : a \in \operatorname{cl}(A) \} = \eta >0$. Let $w = \max \{ |a-Ma| : a \in \operatorname{cl}(A)\} < \infty$. For all $a \in \operatorname{cl}(A)$, we have $| \varphi_\lambda(a) - \varphi_{\lambda+\nu}(a)| \le \nu w$, so we would have $|y-\varphi_{t+ \frac \eta {2w}}(a)| \ge \left| | y-\varphi_t(a)| - |\varphi_t(a) - \varphi_{t+ \frac \eta {2w}}(a) | \right| \ge \frac \eta 2$.
Therefore, we would have $t < t + \frac \eta {2w} \le \inf \{ \lambda \in [0,1] : y \in \varphi_\lambda(\operatorname{cl}(A))\}$, which is a contradiction to the definition of $t$. By the case distinction, $\varphi_t:\operatorname{cl}(A) \to \varphi_{t}(\operatorname{cl}(A))$ is bijective.

Next, we show that $y \in \partial \varphi_t(\operatorname{cl}(A))$. If $y \in \partial A$, we are done. Assume otherwise that $y \notin \partial A$. Since $y \notin A$, we have $t>0$. Assume that $y \in \operatorname{int}(\varphi_t(\operatorname{cl}(A)))$. Then there exists $\e >0$ such that $B(y,\e) \subseteq \varphi_t(\operatorname{cl}(A))$. Let $0<\d \le \frac \e {3w}$ such that $0<t-\d \neq \frac 1 2$. Let $U=\varphi_t^{-1}(B(y,\e))$. The function $f=\varphi_{t-\d}\circ \varphi_t^{-1}:B(y,\e) \to \varphi_{t-\d}(U)$ is a homeomorphism. Since homeomorphisms preserve topology and, for all $v \in \varphi_t(\operatorname{cl}(A))$, we have $|v - f(v)| \le \e/3$, we also have $y \in \varphi_{t-\d}(\operatorname{cl}(A))$, which is a contradiction to the definition of $t$.

Since $\varphi_t$ is a homeomorphism, we have that $\partial \varphi_t(\operatorname{cl}(A)) = \varphi_t(\partial A)$, which ends the proof.
\end{itemize}
\end{proof}

\section{Bounding the volume of certain unions of line segments}\label{sec:boundOnVolume}

Recall that, for $x,y \in \R^d$, the line segment from $x$ to $y$ is denoted by $\ell(x,y)$.

Next, we will prove that the volume of the union of line segments from $a$ to $a+t$ for $a \in \partial A$ is bounded by the length of $t$ times the $(d-1)$-dimensional volume of $\partial A$. Together with the results of the last section, this gives a bound on the volume of $A \bigtriangleup (A+t)$, proving Theorem \ref{thm:volSymDiffTrans}.

\begin{lemma}\label{thm:volLineSegsTrans}
Let $A \subset \R^d$ be a bounded set. Let $t \in \R^d$ be a translation vector. Then, $$\mathcal{H}^d \left( \bigcup \{ \ell (a,a+t) : a \in \partial A \} \right) \le |t| \; \mathcal{H}^{d-1}(\partial A).$$
\end{lemma}

\begin{proof}
 We abbreviate $L = \bigcup \{ \ell (a,a+t) : a \in \partial A\}$.
Let $\d>0$ and let $\{ B_j : j \in \N \}$ be a covering of $\partial A$ with $\diam{B_j} \le \d$ for all $j \in \N$.

For each $j \in \N$, we define a cylinder $Z_j$ such that $L \subseteq \bigcup \{ Z_j : j \in \N \}$. The top and bottom of the cylinder are formed by copies of $B_j$ projected to the orthogonal space of $t$. See Figure \ref{fig:transZylinder}. The bottom of the cylinder $Z_j^b$ sits in the hyperplane that contains a point of $\operatorname{cl}(B_j)$, but does not contain any point of $\operatorname{cl}(B_j)$, when translated in direction $-t$ by any small amount. The top of the cylinder $Z_j^t$ is formed by $Z_j^b + (1+\diam(B_j)) t$. By construction, the cylinder $Z_j$ contains $\bigcup \{ \ell(b,b+t) : b \in B_j \}$.

\begin{figure}[ht]
\begin{center}
\includegraphics[page=1,width=0.6\textwidth]{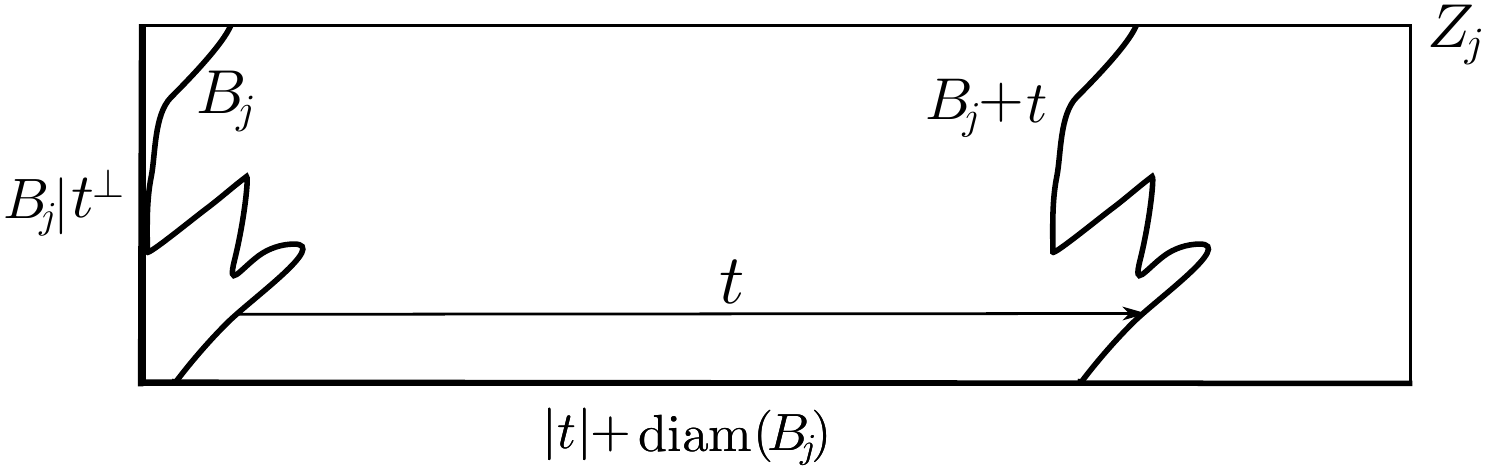}
\caption{The definition of the cylinder $Z_j$, which contains all line segments $\ell(b,b+t)$ for $b \in B_j$.}
\label{fig:transZylinder}
\end{center}
\end{figure}

The volume $\mathcal{H}^d (Z_j) = \mathcal{H}^{d-1} (B_j | t^\bot) \, (1+\diam(B_j)) \, |t|$ (Theorem 3.2.23 in \cite{federer}). Note that $\diam(B_j | t^\bot) \le \diam(B_j)$. By Theorem \ref{thm:isodiametricIneq}, $\mathcal{H}^{d-1}(B_j | t^\bot) \le \omega_{d-1} (\diam(B_j)/2)^{d-1}$. We have 
$$\mathcal{H}^d(L) \, \le \,  \sum_{j \in \N} \mathcal{H}^d (Z_j)
 \, \le \, (1 + \d) |t| \, \sum_{j \in \N} \omega_{d-1} (\diam(B_j)/2)^{d-1}$$
This implies that $\mathcal{H}^d(L) \le (1 + \d) \, |t| \, \mathcal{H}_\d^{d-1} (\partial A)$ for all $\d>0$. Therefore, $\mathcal{H}^d(L) \le |t| \, \mathcal{H}^{d-1} (\partial A)$.
\end{proof}

Next, we will bound the volume of the union of line segments from $a$ to $Ma$ for $a \in \partial A$ for a rotation matrix $M$. Together with the results of the previous section, this gives a bound on the volume of $A \bigtriangleup MA$, proving Theorem \ref{thm:volSymDiffRot}.

\begin{lemma}\label{thm:volLineSegsRot}
Let $A \subset \R^d$ be a bounded set.
Let $M \in \R^{d \times d}$ be a rotation matrix and let $w = \max_{a \in \partial A} |a - Ma|$.
Then, $$\mathcal{H}^d \left( \bigcup \{ \ell(a,Ma) : a \in \partial A \} \right) \le \left( \frac{2d}{d+1}\right)^{\frac {d-1} 2} \, w \, \mathcal{H}^{d-1}(\partial A).$$
\end{lemma}

\begin{proof}
The proof is similar to the one for translations. We abbreviate $L = \bigcup \{ \ell(a,Ma) : a \in \partial A\}$.
Let $\d>0$ and let $\{ B_j : j \in \N \}$ be a covering of $\partial A$ with $\diam{B_j} \le \d$ for all $j \in \N$. Let $w=\max_{a \in \partial A} |a-Ma|$. We again cover $L$ by a set of cylinders $\{ Z_j : j \in \N \}$, which are defined using the $B_j$s. Since the line segments in $L$ are not parallel, top and bottom of $Z_j$ have a volume that is larger than a ball of diameter $B_j$, and therefore we get a constant in the inequality, which is larger than one.

Let $enc(B_j)$ and $enc(MB_j)$ be the smallest enclosing balls of $B_j$ and $MB_j$. Clearly, both have the same radius. By Theorem \ref{thm:Jung}, their radius $r$ is at most $\sqrt{\frac d {2d+2}} \diam(B_j) $. Let $t$ be the vector from the center of $enc(B_j)$ to the center of $enc(MB_j)$. The length $|t| \le w + 2 \sqrt{\frac d {2d+2}} \diam(B_j)$. See Figure \ref{fig:RMZylinder}.

\begin{figure}[ht]
\begin{center}
\includegraphics[page=1,width=0.55\textwidth]{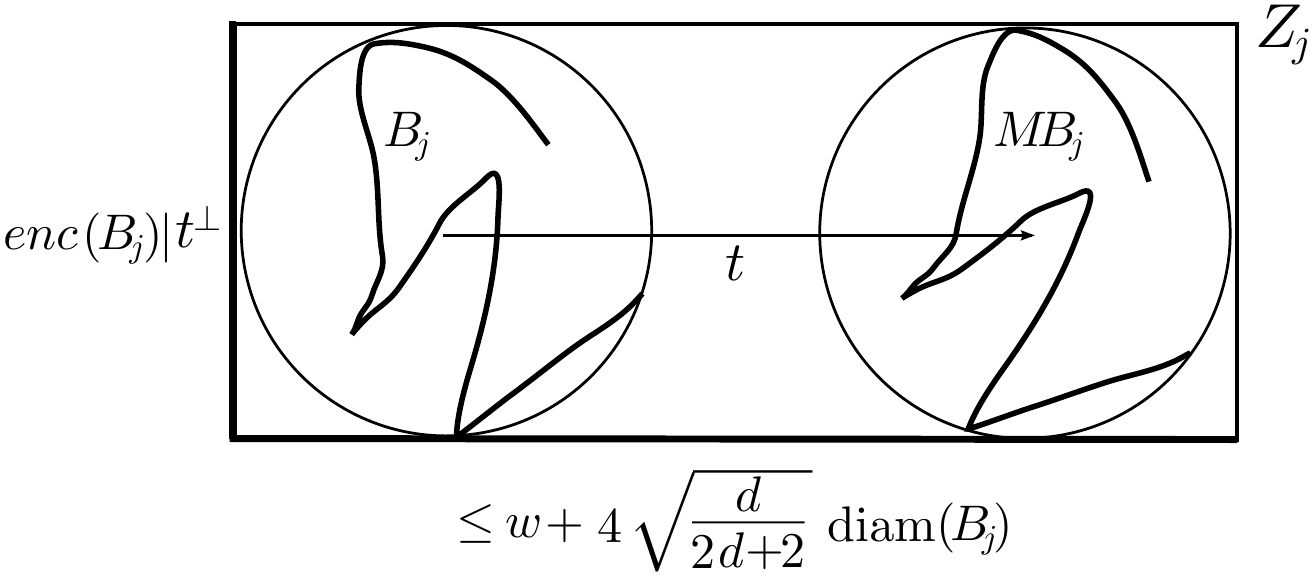}
\caption{The definition of the cylinder $Z_j$, which contains all line segments $\ell(b,Mb)$ for $b \in B_j$.}
\label{fig:RMZylinder}
\end{center}
\end{figure}

Already the convex hull of $enc(B_j)$ and $enc(MB_j)$ contains the union of all line segments $\ell(b,Mb)$ for $b \in B_j$. We again enlarge the set by considering the cylinder $Z_j$ that has copies of the $(d-1)$-dimensional ball $enc(B_j)|t^\bot$ as top and bottom. Top and bottom are touching $enc(B_j)$ and $enc(MB_j)$, respectively, so that the cylinder contains the convex hull of $enc(B_j)$ and $enc(MB_j)$.  The volume of top and bottom equals $\mathcal{H}^{d-1}(enc(B_j)|t^\bot) = \omega_{d-1} r^{d-1} \le \omega_{d-1} (\frac d {2d+2})^{\frac {d-1} 2} \diam(B_j)^{d-1}$. The distance of top and bottom is at most $|t| \le w + 4 \sqrt{\frac d {2d+2}} \diam(B_j)$. By Theorem 3.2.23 in \cite{federer}, the volume of $Z_j$ can be computed as the product of the area of the bottom and the height of the cylinder.

We have
\begin{eqnarray*}
\mathcal{H}^d (L) &\le & \sum_{j \in \N} \mathcal{H}^d(Z_j)\\ 
& \le &   \sum_{j \in \N} \omega_{d-1} \Bigl(\frac d {2d+2}\Bigr)^{\frac {d-1} 2} \diam(B_j)^{d-1} \Bigl(w + 4 \sqrt{\frac d {2d+2}} \diam(B_j) \Bigr)\\ 
& \le & \Bigl(w + 4 \sqrt{\frac d {2d+2}} \d \Bigr) \; \sum_{j \in \N} \omega_{d-1} \Bigl(\frac d {2d+2}\Bigr)^{\frac {d-1} 2} \diam(B_j)^{d-1}\\ 
& \le & 2^{\frac {d-1} 2} \Bigl(\frac d {d+1}\Bigr)^{\frac {d-1} 2} \; \Bigl(w + 4 \sqrt{\frac d {2d+2}} \d \Bigr) \; \sum_{j \in \N} \omega_{d-1} 2^{-(d-1)} \diam(B_j)^{d-1}
\end{eqnarray*}
This implies $\mathcal{H}^d (L) \le (\frac {2d} {d+1})^{\frac {d-1} 2} \; (w + 4 \sqrt{\frac d {2d+2}} \d ) \; \mathcal{H}_\d^{d-1}(\partial A)$ for all $\d >0$ and therefore $\mathcal{H}^d (L) \le (\frac {2d} {d+1})^{\frac {d-1} 2} \; w \; \mathcal{H}^{d-1}(\partial A)$.
\end{proof}

If we could assume in the proof that the covering $\{B_j\}_{j \ge 0}$ contains only balls, then the constant $\left(\frac{2d}{d+1}\right)^{\frac {d-1} 2}$ in Theorem \ref{thm:volLineSegsRot} could be replaced by one. We can do this for sets for which the Hausdorff and the spherical measure coincide. Therefore, we get

\begin{cor}\label{cor:rot1}
Let $A \subset \R^d$ be a bounded set such that $\partial A$ is $(\mathcal{H}^{d-1},d-1)$-rectifiable.
Let $M \in \R^{d \times d}$ be a rotation matrix and let $w = \max_{a \in \partial A} |a - Ma|$.
Then, $$\mathcal{H}^d \left( \bigcup \{ \ell(a,Ma) : a \in \partial A \} \right) \le w \, \mathcal{H}^{d-1}(\partial A).$$
\end{cor}

For example, if $A \subset \R^d$ is the finite union of simplices, then $\partial A$ is $(\mathcal{H}^{d-1},d-1)$-rectifiable. Finite unions of simplices are a common representation of shapes.

\section{Application to shape matching}\label{sec:shapeMatching}

For two shapes $A,B \subset \R^d$, let $F$ be the function that maps a rigid motion $r$ to the volume of overlap of $r(A)$ and $B$. Maximizing this function over all rigid motions is a shape matching problem, as described in the introduction. We are interested in studying the smoothness of the function $F$. In particular, we want to bound $|F(r) -F(s)|$ if the rigid motions $r$ and $s$ are close, meaning that they do not move points from $A$ too far apart. We first prove an easy proposition, which we then use, together with the results of the previous section, to establish the bound that we are interested in.

\begin{prop}\label{prop:symDiff}
For a measure space $(\mathcal{M},\mu)$ and $\mu$-measurable sets $D,E,G$ in $\mathcal{M}$ such that $\mu(D)=\mu(G)$ is $|\mu(D \cap E) - \mu (G \cap E)| \le \mu (D \setminus G)$.
\end{prop}
\begin{proof}
\begin{align*}|\mu(D \cap E) - \mu(G \cap E)| &= |\mu((D \setminus G) \cap E) - \mu((G \setminus D) \cap E)| \\ &\le \max \{ \mu((D\setminus G) \cap E), \mu((G \setminus D) \cap E) \} \\ &\le \max \{ \mu(D\setminus G), \mu(G \setminus D) \} \\ &= \frac 1 2 \, \mu(D \bigtriangleup G)
\end{align*}
\end{proof}

\begin{proof}[Proof of Corollary \ref{cor:shapeMatching}]
Let $r=(M,p)$ and $s=(N,q)$ be rigid motions, and let $w = \max_{a \in \partial A} |Ma - Na|$. Using Proposition \ref{prop:symDiff}, we get
$$| \mathcal{H}^d(r(A) \cap B) - \mathcal{H}^d(s(A) \cap B) | \le \frac 1 2 \, \mathcal{H}^d \bigl(r(A) \bigtriangleup s(A) \bigr) = \frac 1 2 \, \mathcal{H}^d \bigl((s^{-1}\circ r)(A) \bigtriangleup A \bigr).$$
The map $s^{-1}\circ r$ is a rigid motion with rotation matrix $N^{-1}M$ and translation vector $N^{-1}(p-q)$. Therefore, $| \mathcal{H}^d(r(A) \cap B) - \mathcal{H}^d(s(A) \cap B) | \le \frac 1 2 \, \bigl(w + |p-q| \bigr) \, \mathcal{H}^{d-1}(\partial A)$ by Corollary \ref{cor:volSymDiffRM}.
\end{proof}

\section*{Acknowledgements}
I am grateful to Richard Gardner and Gabriele Bianchi for pointing out relations to the covariogram and other problems. With Felix Breuer and Arne M\"uller, I had fruitful discussions on the topic. I thank Helmut Alt and Matthias Henze for reading a preliminary version of the manuscript carefully.


\bibliographystyle{plain}
\bibliography{daria}

\end{document}